\documentclass{amsart}

\usepackage[all]{xy}

\usepackage{latexsym}

\usepackage{amssymb}
\usepackage{amsfonts}
\usepackage{amscd}
\usepackage{amsmath,amsthm}

\newtheorem{lemma}{Lemma}[section]

\newtheorem{theorem}[lemma]{Theorem}

\newtheorem{prop}[lemma]{Proposition}

\newtheorem{cor}[lemma]{Corollary}

{

}

\theoremstyle{definition}

\newtheorem{example}[lemma]{Example}
\newtheorem{definition}[lemma]{\sl Definition}

\theoremstyle{remark}

\numberwithin{equation}{section}

\begin{document}

\title{Witt's theorem for noncommutative conic curves}
\author{A. Nyman}
\address{Department of Mathematics, 516 High St, Western Washington University, Bellingham, WA 98225-9063}
\email{adam.nyman@wwu.edu}

\begin{abstract}
Let $k$ be a field.  We extend the main result in \cite{tsen} to show that {\it all} homogeneous noncommutative curves of genus zero over $k$ are noncommutative $\mathbb{P}^{1}$-bundles over a (possibly) noncommutative base.  Using this result, we  compute complete isomorphism invariants of homogeneous noncommutative curves of genus zero, allowing us to generalize a theorem of Witt.
\end{abstract}

\maketitle

\pagenumbering{arabic}

\section{Introduction}
Throughout this paper, $k$ will denote a field over which everything is defined.  In particular, all categories and equivalences of categories will be $k$-linear, and all bimodules will be $k$-central.

Recall that if $X$ is a smooth projective curve of genus zero with canonical bundle $\omega_{X}$ and a $k$-rational point $P$, then there is a commutative diagram of isomorphisms
\begin{equation} \label{eqn.comdig}
\xymatrix{
X \ar[rr]  \ar[d] & &  \operatorname{Proj }(\bigoplus_{j}H^{0}(X,{\omega_{X}^{*}}^{\otimes j})) \\
\mathbb{P}^{1} \cong \operatorname{Proj }(\bigoplus_{i}H^{0}(X,{\mathcal{O}_{X}(P)}^{\otimes i})) \ar[rr] & & \operatorname{Proj }(\bigoplus_{j}H^{0}(X,{\mathcal{O}_{X}(P)}^{\otimes 2j})) \ar[u]
}
\end{equation}
whose top horizontal is the anti-canonical embedding, whose left vertical is the embedding induced by the invertible sheaf $\mathcal{O}_{X}(P)$, whose bottom horizontal is the Veronese embedding, and whose right vertical is induced by an isomorphism $\mathcal{O}_{X}(P)^{\otimes 2} \cong \omega_{X}^{*}$.

The purpose of this paper is to show that a version of (\ref{eqn.comdig}) exists when $X$ is any {\it noncommutative} curve of genus zero (see (\ref{eqn.ncdig}) below), and to use this fact to generalize a theorem of Witt to the noncommutative context (Corollary \ref{cor.newwitt}).  It will follow that even if $X$ is a smooth commutative curve of genus zero without a $k$-rational point, it is isomorphic to a noncommutative projective line, in contrast to what happens in the purely commutative situation.  This illustrates the increased flexibility present in the noncommutative setting.

In order to make sense of the noncommutative version of (\ref{eqn.comdig}), we proceed to describe noncommutative versions of the commutative objects in (\ref{eqn.comdig}).  A {\it noncommutative curve of genus zero} \cite{kussin} is a small abelian, noetherian category $\sf H$ which has a tilting bundle, an object of infinite length, an Auslander-Reiten translation and has the property that for all simple objects $\mathcal{S}$ in ${\sf H}$, we have $\operatorname{Ext}_{{\sf H}}^{1}(\mathcal{S},\mathcal{S}) \neq 0$.  In \cite{kussin}, this last condition makes the curve  {\it homogeneous}, but we shall omit this term for simplicity.  The motivating example is the category of coherent sheaves over a smooth projective curve of genus zero.

Next, we recall the definition of a noncommutative version of $\mathbb{P}^{1}$ due to M. Van den Bergh \cite{vandenbergh}.  If $K$ and $L$ are finite extensions of $k$ and $N$ is a $K-L$-bimodule of finite dimension as both a $K$-module and an $L$-module, then one can form the $\mathbb{Z}$-algebra $\mathbb{S}^{nc}(N)$, the noncommutative symmetric algebra of $N$.  The {\it noncommutative $\mathbb{P}^{1}$-bundle generated by $N$}, $\mathbb{P}^{nc}(N)$, is defined to be the quotient of the category of graded right $\mathbb{S}^{nc}(N)$-modules modulo the full subcategory of direct limits of right bounded modules (see \cite{stafford} for more details on this quotient category construction).  Below, we will write ${\sf coh }\mathbb{P}^{nc}(N)$ for the full subcategory of noetherian objects.

In order to describe the noncommutative analogue of the isomorphism in the lower left corner of (\ref{eqn.comdig}), we first recall the main result of \cite{tsen}, which relates the concept of a noncommutative curve of genus zero to the notion of noncommutative $\mathbb{P}^{1}$-bundle.  To this end, we state some preliminaries.  Let ${\sf H}$ be a noncommutative curve of genus zero.  The full subcategory ${\sf H}_{0}$ of objects of ${\sf H}$ of finite length is a Serre subcategory, and the quotient category ${\sf H}/{\sf H}_{0}$ is equivalent to ${\sf mod }k({\sf H})$, the category of right modules over a division ring $k({\sf H})$ \cite[Section 1]{lenzing}.  The {\it rank} of an object in ${\sf H}$ is defined to be the dimension of the image of the object under the quotient functor.  A {\it vector bundle} in ${\sf H}$ is an object without a simple subobject, and a {\it line bundle} is a rank one vector bundle.  Suppose $\mathcal{L}$ is a line bundle on ${\sf H}$.  Then there exists an indecomposable bundle $\overline{\mathcal{L}}$ and an irreducible morphism $\mathcal{L} \longrightarrow \overline{\mathcal{L}}$ coming from an AR sequence starting at $\mathcal{L}$.  Following \cite{kussin}, we let $M$ denote the {\it underlying bimodule of $\sf H$}, defined as ${}_{\operatorname{End }(\overline{\mathcal{L}})}\operatorname{Hom}_{{\sf H}}(\mathcal{L},\overline{\mathcal{L}})_{\operatorname{End }(\mathcal{L})}$.  The main result in \cite{tsen} is that if $\operatorname{End}(L)$ and $\operatorname{End}(\overline{L})$ are {\it commutative}, then
\begin{equation} \label{eqn.mainequiv}
\tilde{H} \equiv \mathbb{P}^{nc}(M)
\end{equation}
where $\tilde{H}$ is the unique locally noetherian category whose subcategory of noetherian objects is ${\sf H}$ and $\equiv$ denotes equivalence of categories.

There are many interesting noncommutative curves of genus zero not described by this result.  For example, if $X$ is a (commutative) nondegenerate conic without a $k$-rational point, then the indecomposable bundle $\overline{\mathcal{L}}$ has rank two, and $\operatorname{End}(\mathcal{\overline{L}})$ is a non-split quaternion algebra.  Furthermore, there exists a degree two extension $k'$ of $k$ such that $\operatorname{End}(\mathcal{\overline{L}}) \otimes_{k}k' \cong M_{2}(k')$ and $C_{k'} \cong \mathbb{P}_{k'}^{1}$ \cite[Remark 3.7]{conics}.

In this paper, we first show, in Section \ref{section.defn}, that the notion of $\mathbb{P}^{1}$-bundle still makes sense when the finite extensions $K$ and $L$ of $k$ are replaced by arbitrary noetherian $k$-algebras.  We then observe, in Section \ref{section.tsen}, that the proof of \cite[Theorem 3.10]{tsen} shows that the equivalence (\ref{eqn.mainequiv}) still holds in case $\operatorname{End}(L)$ and $\operatorname{End}(\overline{L})$ are division rings finite dimensional over $k$.  It will follow that (\ref{eqn.mainequiv}) applies to all noncommutative curves of genus zero, yielding the noncommutative analogue of the lower left isomorphism in (\ref{eqn.comdig}).  In fact, we have the following result (see Sections \ref{section.defn} and \ref{section.tsen} for more details):  If ${\sf H}$ is a noncommutative curve of genus zero with underlying bimodule $M$, then there exists a $\mathbb{Z}$-algebra $C$ such that the diagram
\begin{equation} \label{eqn.ncdig}
\xymatrix{
{\sf H} \ar[rr]  \ar[d] & &  \sf{proj }(\bigoplus_{j \geq 0}\operatorname{Hom}_{\sf H}(\tau^{j}\mathcal{L},\mathcal{L})) \\
{\sf coh }\mathbb{P}^{nc}(M) \equiv {\sf{proj }}(\bigoplus_{ij} C_{ij}) \ar[rr] & & {\sf{proj }}(\bigoplus_{ij}C_{2i2j}) \ar[u]
}
\end{equation}
whose top horizontal is the noncommutative anticanonical embedding, whose left vertical is the equivalence in \cite[Theorem 11.1.1]{stafford}, whose bottom horizontal is the $\mathbb{Z}$-algebra Veronese, and whose right vertical is induced from the 1-periodicity of the Veronese of $C$, commutes.

As a consequence of our generalization of (\ref{eqn.mainequiv}), we import from \cite{vandenbergh} an intrinsic notion of base change for noncommutative curves of genus zero, which allows us to generalize \cite[Remark 3.7]{conics} (see Example \ref{example.quaternion}).  We then use (\ref{eqn.mainequiv}) to classify noncommutative curves of genus zero up to isomorphism (Theorem \ref{thm.equivs}), generalizing \cite[Theorem 5.2]{nyman} and \cite[Propostion 5.1.4]{kussin}:

\begin{theorem}
For $i=1,2$, let $D_{i}$ and $E_{i}$ be division rings finite dimensional over $k$, let $M$ be a $D_{1}-D_{2}$-bimodule of left right dimension $(2,2)$ or $(1,4)$ and let $N$ be an $E_{1}-E_{2}$-bimodule of left right dimension $(2,2)$ or $(1,4)$.  There is an equivalence
$$
{\sf H}_{1} = \mathbb{P}^{nc}(M) \longrightarrow \mathbb{P}^{nc}(N) = {\sf H}_{2}
$$
if and only if either
\begin{enumerate}
\item{} there exist isomorphisms $\phi_{i}:D_{i} \rightarrow E_{i}$ and $\psi:M \rightarrow N$ such that
$$
\psi(a \cdot m \cdot b) = \phi_{1}(a) \cdot \psi(m) \cdot \phi_{2}(b),
$$
or

\item{} there exist isomorphisms $\phi_{1}:D_{1} \rightarrow E_{2}$, $\phi_{2}:D_{2} \rightarrow E_{1}$ and $\psi:M \rightarrow N^{*}$ such that
$$
\psi(a \cdot m \cdot b) = \phi_{1}(a) \cdot \psi(m) \cdot \phi_{2}(b).
$$
\end{enumerate}
\end{theorem}

Specializing this result allows us to classify noncommutative conics (that is, noncommutative curves of genus zero over $k$ with underlying $(1,4)$-bimodule, i.e. whose underlying bimodule is one-dimensional as a left module and four-dimensional as a right module) in terms of their underlying bimodules (Corollary \ref{cor.newwitt}).  As we shall prove in Corollary \ref{cor.wittcor}, this is a generalization of a theorem of Witt which states that if $\operatorname{char }k \neq 2$, then two quaternion $k$-algebras are isomorphic over $k$ if and only if their associated conics are isomorphic.

{\it Acknowledgement:}  I thank D. Kussin for his helpful comments.

\section{Noncommutative symmetric algebras over noetherian rings and noncommutative $\mathbb{P}^{1}$-bundles} \label{section.defn}

\subsection{Duality and admissible bimodules}
Let $R$ and $S$ be noetherian $k$-algebras, and let ${\sf Mod }R$ and ${\sf Mod }S$ denote their categories of right modules.  In this section, following \cite{vandenbergh}, we define the noncommutative symmetric algebra of certain $R-S$-bimodules.  Some of the exposition is adapted from \cite[Section 2]{tsen}.

We assume throughout this section that $N$ is an $R-S$-bimodule which is finitely generated projective as both a left $R$-module and as a right $S$-module.  We next recall the notion of left and right dual of $N$.  The {\it right dual of $N$}, denoted $N^{*}$, is the $S-R$-bimodule whose underlying set is
$\operatorname{Hom}_{S}(N_{S},S)$, with action
$$
(a \cdot \psi \cdot b)(n)=a\psi(bn)
$$
for all $\psi \in \operatorname{Hom}_{S}(N_{S},S)$, $a\in S$ and $b \in R$.

The {\it left dual of $N$}, denoted ${}^{*}N$, is the $S-R$-bimodule whose underlying set is
$\operatorname{Hom}_{R}({}_{R}N,R)$, with action
$$
(a \cdot \phi
\cdot b)(n)=\phi(na)b
$$
for all $\phi \in \operatorname{Hom}_{R}({}_{R}N,R)$, $a \in S$ and $b \in R$.  This assignment extends to morphisms between $R-S$-bimodules in the obvious way.

We set
$$
N^{i*}:=
\begin{cases}
N & \text{if $i=0$}, \\
(N^{i-1*})^{*} & \text{ if $i>0$}, \\
{}^{*}(N^{i+1*}) & \text{ if $i<0$}.
\end{cases}
$$
We note that although $N^{*}$ (resp. ${}^{*}N$) is finitely generated projective on the left (resp. finitely generated projective on the right), it is not clear that $N^{*}$ is finitely generated projective on the right (resp. finitely generated projective on the left).  Therefore, we make the following

\begin{definition}
We say $N$ is {\it admissible} if $N^{i*}$ is finitely generated projective on each side for all $i \in \mathbb{Z}$.
\end{definition}
We remark that if $R$ and $S$ are finite dimensional simple rings over $k$, then $N$ is automatically admissible.  Furthermore, if $L$ is an $R-S$-bimodule such that $-\otimes_{R}L :{\sf Mod }R \rightarrow {\sf Mod }S$ is an equivalence, then $L$ is admissible by Morita theory.

Next, let
$$
F_{i} = \begin{cases}
R & \mbox{if $i$ is even, and} \\
S & \mbox{if $i$ is odd.}
\end{cases}
$$
We check that, if $N$ is admissible, then for each $i$, both pairs of functors
\begin{equation} \label{eqn.adjointone}
(-\otimes_{F_{i}} N^{i*},-\otimes_{F_{i+1}} N^{i+1*})
\end{equation}
and
\begin{equation} \label{eqn.adjointtwo}
(-\otimes_{F_{i}} {}^{*}(N^{i+1*}), -\otimes_{F_{i+1}} N^{i+1*})
\end{equation}
between the category of right $F_{i}$-modules and the category of right $F_{i+1}$-modules have canonical adjoint structures.  

To prove that (\ref{eqn.adjointone}) has a canonical adjoint pair structure, we note that since $N^{i*}_{F_{i+1}}$ is finitely generated projective, the functor $\operatorname{Hom}_{F_{i+1}}(N^{i*}_{F_{i+1}},-) :{\sf Mod }F_{i+1} \rightarrow {\sf Mod }F_{i}$ commutes with direct sums and is right exact.  Therefore, by the Eilenberg-Watts theorem, there is an isomorphism of functors
\begin{equation} \label{eqn.psi}
\psi:\operatorname{Hom}_{F_{i+1}}(N^{i*}_{F_{i+1}},-) \rightarrow -\otimes_{F_{i+1}}\operatorname{Hom}_{F_{i+1}}(N^{i*}_{F_{i+1}},S).
\end{equation}
The isomorphism $\psi$, together with the canonical adjoint pair structure on
$$
(-\otimes_{F_{i}} N^{i*}, \operatorname{Hom}_{F_{i+1}}(N^{i*}_{F_{i+1}},-))
$$
endows the pair (\ref{eqn.adjointone}) with an adjoint pair structure.

Next we describe a canonical adjoint pair structure for (\ref{eqn.adjointtwo}).  To this end, we leave it as a straightforward exercise to check that the function
\begin{equation} \label{eqn.newpsi}
N^{i*} \longrightarrow {}^{*}(N^{i+1*}) = \operatorname{Hom}_{F_{i+1}}({}_{F_{i+1}}\operatorname{Hom}_{F_{i+1}}(N^{i*}_{F_{i+1}},F_{i+1}), F_{i+1})
\end{equation}
defined by sending $n$ to evaluation at $n$, is an $F_{i}-F_{i+1}$-bimodule isomorphism.  The adjoint structure on (\ref{eqn.adjointtwo}) is then induced from that of (\ref{eqn.adjointone}) and the isomorphism (\ref{eqn.newpsi}).

We denote the images of the units applied to $F_{i}$ by $Q_{i}$ and $Q_{i}'$, respectively.

The proof of the following fact, which is implicit in \cite{vandenbergh}, is almost identical to \cite[p. 201-202]{nyman}.
\begin{lemma} \label{lemma.dualityandtensor}
Suppose $N$ is admissible, $L$ is an $S-R$-bimodule and $P_{1}$ and $P_{2}$ are $R-S$-bimodules such that
$$
(-\otimes_{S}L, -\otimes_{R}P_{1})
$$
and
$$
(-\otimes_{R}P_{2}, - \otimes_{S}L)
$$
are adjoint pairs between ${\sf Mod }R$ and ${\sf Mod }S$.  Then
$$
(N \otimes_{S}L)^{*} \cong P_{1} \otimes_{S} N^{*}
$$
and
$$
{}^{*}(N \otimes_{S} L) \cong P_{2} \otimes_{S} {}^{*}N.
$$
A similar formula holds for the right and left duals of $L \otimes_{R} N$.
\end{lemma}
Finally, suppose $k'$ is an extension field of $k$.  We let $N_{k'}$ denote the $k'$-central $R\otimes_{k}k' - S \otimes_{k}k'$-bimodule whose underlying set is $N \otimes_{k}k'$ and whose action is inherited from the $R-S$-bimodule action on $N$.  We note that since $R$ and $S$ are noetherian, $N$ is finitely presented on each side, so base change on bimodules is compatible with taking duals (as in \cite[Lemma 3.1.9]{vandenbergh}).  Therefore, if $N$ is admissible, so is $N_{k'}$.

\subsection{Free admissibility}
In the proof of Theorem \ref{thm.newtsen} we will need to adapt the proof of \cite[Theorem 3.10]{tsen}.  This, in turn, will require us, by \cite[Lemma 3.3]{tsen}, to prove that the images of the unit maps of the adjoint pairs (\ref{eqn.adjointone}) and (\ref{eqn.adjointtwo}) take on a particular form in case $N^{i*}$ is {\it free} of finite rank on either side for all $i$.  Since $R$ and $S$ are noetherian, they are IBN rings \cite[p. 216]{cohn}, so that the notion of rank of a left or right free module is well defined.

\begin{definition}
We say $N$ is {\it free admissible} if $N^{i*}$ is free of finite rank on each side for all $i \in \mathbb{Z}$.
\end{definition}

We assume throughout this section that $N$ is a free admissible $R-S$-bimodule.  We note that this hypothesis holds if $R$ and $S$ are simple rings finite dimensional over $k$ and $N$ is finitely generated on either side.

We suppose $N^{i*}_{F_{i+1}}$ is free of rank $n$, $\{\phi_{1},\ldots,\phi_{n}\}$ is a right basis for $N^{i*}$ and $\{f_{1},\ldots, f_{n}\}$ is the corresponding dual left basis for $N^{i+1*}$.  We suppose $\{\phi'_{1},\ldots,\phi'_{n}\}$ is the right basis for ${}^{*}(N^{i+1*})$ dual to $\{f_{1},\ldots, f_{n}\}$.

We define $\eta_{i}: F_{i} \longrightarrow N^{i*} \otimes_{F_{i+1}}  N^{i+1*}$ by
\begin{equation} \label{eqn.unit1}
\eta_{i}(a) = a \sum_{j} \phi_{j} \otimes f_{j}
\end{equation}
and we define $\eta_{i}':F_{i} \longrightarrow {}^{*}(N^{i+1*}) \otimes_{F_{i+1}} N^{i+1*}$ by
\begin{equation} \label{eqn.unit2}
\eta_{i}'(a) = a \sum_{j} \phi'_{j} \otimes f_{j}.
\end{equation}

\begin{prop} \label{prop.newreln}
If $i \in \mathbb{Z}$, then the images of the units of the adjoint pairs (\ref{eqn.adjointone}) and (\ref{eqn.adjointtwo}) applied to $F_{i}$ equal the images of $\eta_{i}$ and $\eta_{i}'$ respectively.
\end{prop}

\begin{proof}
We prove the first assertion.  The proof of the second follows easily from the first by the definition of (\ref{eqn.newpsi}).  Let
$$
(-\otimes_{F_{i}} N^{i*}, \operatorname{Hom}_{F_{i+1}}(N^{i*}_{F_{i+1}},-), \eta, \epsilon)
$$
denote the canonical adjunction between ${\sf Mod }F_{i}$ and ${\sf Mod }F_{i+1}$.  If $P$ is a right $F_{i+1}$-module, then the map $\psi_{P}$ from (\ref{eqn.psi}),
$$
\psi_{P}:\operatorname{Hom}_{F_{i+1}}(N^{i*}_{F_{i+1}},P) \rightarrow P \otimes_{F_{i+1}}N^{i+1*}
$$
is defined by $\psi_{P}(\delta) := \sum_{j}\delta(\phi_{j}) \otimes f_{j}$.  Thus, the adjunction (\ref{eqn.adjointone}) is
$$
(-\otimes_{F_{i}} N^{i*}, -\otimes_{F_{i+1}}N^{i+1*}, \tilde{\eta}, \tilde{\epsilon})
$$
with $\tilde{\eta} = (\psi*(-\otimes N^{i*})) \circ \eta$ and $\tilde{\epsilon} = \epsilon \circ ((- \otimes N^{i*}) * \psi^{-1})$ where $*$ denotes the horizontal composition of natural transformations.  It remains to explicitly compute $\tilde{\eta}_{F_{i}}(1)$ and show it may be identified with $\eta_{i}(1)$.  We leave this elementary computation to the reader.
\end{proof}

\subsection{Noncommutative symmetric algebras and $\mathbb{P}^{1}$-bundles}
The purpose of this section is to define the notions of noncommutative symmetric algebra and noncommutative $\mathbb{P}^{1}$-bundle of an admissible bimodule, following M. Van den Bergh.  We first introduce a convention that will be in effect throughout the remainder of this paper:  all unadorned tensor products will be bimodule tensor products over the appropriate base ring.

We next recall the definition of $\mathbb{Z}$-algebra from \cite[Section 2]{quadrics}: a $\mathbb{Z}$-{\it algebra} is a ring $A$ with decomposition $A=\oplus_{i,j \in \mathbb{Z}}A_{ij}$ into $k$-vector spaces, such that multiplication has the property $A_{ij}A_{jk} \subset A_{ik}$ while $A_{ij}A_{kl}=0$ if $j \neq k$.  Furthermore, for $i \in \mathbb{Z}$, there is a local unit $e_{i} \in A_{ii}$, such that if $a \in A_{ij}$, then $e_{i}a=a=ae_{j}$.  Just as with $\mathbb{Z}$-graded algebras, we define ${\sf proj }A$ to be the quotient category ${\sf gr }A/{\sf tors }A$, where ${\sf gr }A$ denotes the category of finitely generated graded right $A$-modules and ${\sf tors }A$ denotes the full subcategory of ${\sf gr }A$ consisting of right bounded modules.

We will need the following terminology regarding $\mathbb{Z}$-algebras, from \cite{quadrics}: Let $i \in \mathbb{Z}$.  If $A$ is a $\mathbb{Z}$-algebra, we let $A(i)$ (the shift of $A$ by $i$) denote the $\mathbb{Z}$-algebra with $A(i)_{jk}=A_{i+j,i+k}$, and with multiplication induced from that of $A$.  A $\mathbb{Z}$-algebra is called $i$-periodic if $A \cong A(i)$.  If $B$ is a $\mathbb{Z}$-graded ring, we let $\tilde{B}$ denote the $\mathbb{Z}$-algebra with $\tilde{B}_{ij}=B_{j-i}$.

We have the following
\begin{lemma} \cite[Lemma 2.4]{quadrics} \label{lemma.veronese}
If $A$ is a 1-periodic $\mathbb{Z}$-algebra, then $A \cong \tilde{B}$ for a $\mathbb{Z}$-graded ring $B$.
\end{lemma}
Let $A^{(2)}$ denote the 2-Veronese of $A$, i.e., the $\mathbb{Z}$-subalgebra of $A$ consisting of all even components.  We will need the fact that if $A$ is noetherian (i.e. the category of graded right $A$-modules is locally noetherian), then inclusion $A^{(2)} \rightarrow A$ induces an equivalence ${\sf proj }A \rightarrow {\sf proj }A^{(2)}$ \cite[Lemma 2.5]{quadrics}.

Now we are finally ready to recall (from \cite{vandenbergh}) the definition of the noncommutative symmetric algebra of an admissible bimodule $N$.  The {\it noncommutative symmetric algebra of $N$}, denoted $\mathbb{S}^{nc}(N)$, is the $\mathbb{Z}$-algebra $\underset{i,j \in
\mathbb{Z}}{\oplus}A_{ij}$ with components defined as follows:
\begin{itemize}
\item{} $A_{ij}=0$ if $i>j$.

\item{} $A_{ii}=R$ for $i$ even,

\item{} $A_{ii}=S$ for $i$ odd, and

\item{} $A_{ii+1}=N^{i*}$.
\end{itemize}
In order to define $A_{ij}$ for $j>i+1$, we introduce some notation: we define $T_{ii+1} := A_{ii+1}$, and, for $j>i+1$, we define
$$
T_{ij} := A_{ii+1} \otimes A_{i+1i+2} \otimes \cdots \otimes A_{j-1j}.
$$
We let $R_{i i+1}:= 0$, $R_{i i+2}:=Q_{i}$,
$$
R_{i i+3}:=Q_{i} \otimes N^{i+2*}+N^{i*} \otimes Q_{i+1},
$$
and, for $j>i+3$, we let
$$
R_{ij} := Q_{i} \otimes T_{i+2j}+T_{ii+1}\otimes Q_{i+1} \otimes T_{i+3j}+\cdots + T_{ij-2} \otimes Q_{j-2}.
$$

\begin{itemize}
\item{} For $j>i+1$, we define $A_{ij}$ as the quotient $T_{ij}/R_{ij}$.
\end{itemize}
Multiplication in $\mathbb{S}^{nc}(N)$ is defined as follows:
\begin{itemize}

\item{} if $x \in A_{ij}$, $y \in A_{lk}$ and $j \neq l$, then $xy=0$,

\item{} if $x \in A_{ij}$ and $y \in A_{jk}$, with either $i=j$ or $j=k$, then $xy$ is induced by the usual scalar action,

\item{}  otherwise, if $i<j<k$, we have
\begin{eqnarray*}
A_{ij} \otimes A_{jk} & = & \frac{T_{ij}}{R_{ij}} \otimes \frac{T_{jk}}{R_{jk}}\\
& \cong & \frac{T_{ik}}{R_{ij}\otimes T_{jk}+T_{ij} \otimes
R_{jk}}.
\end{eqnarray*}
Since $R_{ij} \otimes T_{jk}+T_{ij} \otimes R_{jk}$ is a submodule of $R_{ik}$,
there is thus an epi $\mu_{ijk}:A_{ij} \otimes A_{jk} \longrightarrow
A_{ik}$.
\end{itemize}
The following is an easy consequence of \cite[Lemma 6.6]{duality}:
\begin{lemma} \label{lemma.functorality}
If $\psi:M \rightarrow N$ is an isomorphism of admissible $R-S$-bimodules, then $\psi$ induces an isomorphism $\mathbb{S}^{nc}(M) \rightarrow \mathbb{S}^{nc}(N)$.
\end{lemma}

We define $\mathbb{P}^{nc}(N)$, to be the quotient of the category of graded right $\mathbb{S}^{nc}(N)$-modules modulo the full subcategory of direct limits of right bounded modules \cite[Definition 1.1]{vandenbergh}.  For the motivation behind this definition, we refer the reader to \cite[Section 2]{stafford}.

Furthermore, if $k'$ is a field extension of $k$, we define $\mathbb{P}^{nc}(N)_{k'} := \mathbb{P}^{nc}(N_{k'})$.

\subsection{Morita equivalence}
We end this section by studying the behavior of $\mathbb{P}^{nc}(N)$ under Morita equivalence of bases.  We introduce notation that will be used throughout the rest of the section.  Let $S'$ be a $k$-algebra and let $F:{\sf Mod }S \rightarrow {\sf Mod }S'$ be an equivalence.  By the Eilenberg-Watts theorem, we may write
$$
F=-\otimes_{S}L
$$
where $L$ is an $S-S'$-bimodule such that there exists an $S'-S$-bimodule $L'$ such that $L \otimes_{S'} L'$ and $L' \otimes_{S} L$ are trivial.  Similarly, let $R'$ be a $k$-algebra and let $G:R-{\sf Mod } \rightarrow R'-{\sf Mod}$ be an equivalence given by
$$
G=P \otimes_{R} -
$$
where $P$ is an $R'-R$-bimodule such that there exists an $R-R'$-bimodule $P'$ such that $P' \otimes_{R'} P$ and $P \otimes_{R} P'$ are trivial.

\begin{prop} \label{prop.morita}
Suppose $R, R', S$ and $S'$ are noetherian $k$-algebras.  If $N$ is admissible, then so are $N \otimes_{S} L$ and $P \otimes_{R} N$, and there are equivalences
$$
\mathbb{P}^{nc}(N) \rightarrow \mathbb{P}^{nc}(N \otimes_{S} L)
$$
and
$$
\mathbb{P}^{nc}(N) \rightarrow \mathbb{P}^{nc}(P \otimes_{R} N).
$$
In particular, if $R$ and $S$ are division rings finite dimensional over $k$, if there are $k$-algebra isomorphisms $\phi_{1}:R \rightarrow R'$ and $\phi_{2}:S \rightarrow S'$, and if $N'$ is the $R'-S'$-bimodule whose underlying set is $N$ and whose action is induced by $\phi_{1}$ and $\phi_{2}$, then there is an equivalence $\mathbb{P}^{nc}(N) \rightarrow \mathbb{P}^{nc}(N')$.
\end{prop}

\begin{proof}
The first part follows from Lemma \ref{lemma.dualityandtensor} and the fact that the tensor product of two bimodules which are finitely generated projective on each side is finitely generated projective on each side.  It follows that the noncommutative spaces in the conclusion exist.

The proof that the stated equivalences exist is almost identical to the proof of \cite[Lemma 4.1 and Theorem 4.1]{decomp} in light of Lemma \ref{lemma.dualityandtensor}, so we omit the details.

Finally, the last statement follows from the first.
\end{proof}

\section{Tsen's theorem over noncommutative bases and commutative conics} \label{section.tsen}
Throughout the rest of this paper, if ${\sf H}$ is a noncommutative curve of genus zero and $\tau$ is an AR-translation functor on ${\sf H}$, then we let $\tau^{-1}$ denote a right (and left) adjoint to $\tau$.  We say a bimodule $M$ over division rings is an $(m,n)$-bimodule if it has left dimension $m$ and right dimension $n$.

\subsection{Tsen's theorem and an immediate consequence}
Our goal in this section is to prove the following theorem, and to deduce some straightforward consequences.
\begin{theorem} \label{thm.newtsen}
Let $\sf H$ be a noncommutative curve of genus zero with underlying bimodule $M$ and let $\tilde{{\sf H}}$ denote the unique locally noetherian category whose subcategory of noetherian objects is ${\sf H}$. Then there is an equivalence
\begin{equation} \label{eqn.newmainequiv}
\tilde{H} \rightarrow \mathbb{P}^{nc}(M).
\end{equation}
Conversely, if $M$ is a $(2,2)$- or $(1,4)$-bimodule over division rings finite dimensional over $k$, then $\mathbb{P}^{nc}(M)$ is a noncommutative curve of genus zero with $M$ as an underlying bimodule.
\end{theorem}

\begin{proof}
As a consequence of Section \ref{section.defn}, the space $\mathbb{P}^{nc}(M)$ is well defined.  Furthermore, by Proposition \ref{prop.newreln}, the proof of \cite[Theorem 3.10]{tsen} applies to our situation except for the proof of \cite[Proposition 3.8]{tsen} in case $M$ is a $(2,2)$-bimodule, which relies on the commutativity of $\operatorname{End}(\mathcal{L})$ and $\operatorname{End}(\overline{\mathcal{L}})$.  However, it is proven in \cite{newchan} that the commutativity assumption is not necessary.
\end{proof}
In order to prove Theorem \ref{thm.equivs} below, we will need Corollary \ref{cor.mstar}.  To prove it, we will invoke the $\mathbb{Z}$-algebra which we now define.  We let
\begin{equation} \label{eqn.oi}
\mathcal{O}(n) := \begin{cases} \tau^{\frac{-n}{2}} \mathcal{L} & \mbox{if $n$ is even} \\ \tau^{\frac{-(n-1)}{2}} \overline{\mathcal{L}} & \mbox{if $n$ is odd.} \end{cases}
\end{equation}
We define a $\mathbb{Z}$-algebra $C$ by setting
$$
C_{ij} = \begin{cases} \operatorname{Hom}(\mathcal{O}(-j),\mathcal{O}(-i)) & \mbox{if $j \geq i$} \\ 0 & \mbox{if $i>j$}\end{cases}
$$
and defining multiplication as composition.  The $\mathbb{Z}$-algebra $C$ is equal to the algebra $H(1)$ defined in \cite[Section 3]{tsen}.

\begin{lemma} \label{lemma.c}
If $M$ is an underlying bimodule of ${\sf H}$, then there is an isomorphism of $\mathbb{Z}$-algebras $\mathbb{S}^{nc}(M^{*}) \rightarrow C$.  Therefore,
$$
\tilde{\sf H} \equiv {\sf Proj }C \equiv \mathbb{P}^{nc}(M^{*}).
$$
\end{lemma}

\begin{proof}
There is an isomorphism $C(-1) \rightarrow \mathbb{S}^{nc}(M)$ by \cite[Theorem 3.10]{tsen}, which induces an isomorphism $C \rightarrow \mathbb{S}^{nc}(M)(1)$.  Furthermore, by repeated application of \cite[Lemma 2.1]{tsen}, one can show that there is an isomorphism $\mathbb{S}^{nc}(M)(1) \rightarrow \mathbb{S}^{nc}(M^{*})$.

To prove the second result, we note that the sequence $\{\mathcal{O}(n)\}_{n \in \mathbb{Z}}$ is ample in the sense of \cite[Section 3]{tsen} by the proof of \cite[Proposition 3.9]{tsen}.  Therefore, by \cite[Theorem 11.1.1]{stafford}, there is an equivalence $\tilde{\sf H} \equiv {\sf Proj }C$.  The last equivalence follows from the first paragraph.
\end{proof}
The next result follows immediately from Theorem \ref{thm.newtsen} and Lemma \ref{lemma.c}:
\begin{cor} \label{cor.mstar}
If $M$ is either a $(2,2)$- or $(1,4)$-bimodule over division rings finite dimensional over $k$, then there is an equivalence $\mathbb{P}^{nc}(M) \rightarrow \mathbb{P}^{nc}(M^{*})$.
\end{cor}

\subsection{Commutativity of diagram (\ref{eqn.ncdig})}
In this section, we explicitly define diagram (\ref{eqn.ncdig}) and prove that it is commutative.  First recall that if $\sigma$ is an auto-equivalence of ${\sf H}$ then the abelian group $\bigoplus_{j \geq 0}\operatorname{Hom}_{\sf H}(\sigma^{-j}\mathcal{L},\mathcal{L})$ can be made into a $\mathbb{Z}$-graded $k$-algebra by defining multiplication as follows: for $a \in \operatorname{Hom}_{\sf H}(\sigma^{-i}\mathcal{L},\mathcal{L})$ and $b \in \operatorname{Hom}_{\sf H}(\sigma^{-j}\mathcal{L},\mathcal{L})$, we let $a \cdot b := a \circ \sigma^{-i}(b)$.

We let the top functor in (\ref{eqn.ncdig}),
$$
{\sf H} \longrightarrow {\sf proj }(\bigoplus_{j \geq 0}\operatorname{Hom}_{\sf H}(\tau^{j}\mathcal{L},\mathcal{L})),
$$
be defined by sending an object $\mathcal{M}$ to the object in the quotient category corresponding to the graded right module
$$
\bigoplus_{j \geq 0}\operatorname{Hom}_{\sf H}(\tau^{j}\mathcal{L},\mathcal{M})
$$
with obvious action, and by sending a morphism to the usual induced morphism (compare with the equivalence from \cite[Theorem 4.5]{az}).  The fact that this functor is an equivalence will follow from the commutativity of (\ref{eqn.ncdig}).

We define the left vertical to be the equivalence from Lemma \ref{lemma.c}.  In particular, it takes $\mathcal{M}$ to the image (under the quotient functor) of the graded right $C$-module $\bigoplus_{j}\operatorname{Hom}_{\sf H}(\mathcal{O}(-j),\mathcal{M})$ with obvious action by $C$ and with the obvious assignment on morphisms.  The fact that the lower left category in (\ref{eqn.ncdig}) is equivalent to the full subcategory of noetherian objects of $\mathbb{P}^{nc}(M)$ follows from Theorem \ref{thm.newtsen}.

Since the the lower left category is noetherian, its 2-Veronese is an equivalence by the remark following Lemma \ref{lemma.veronese}, and is induced by sending graded modules and morphisms to their even components.  This is the bottom horizontal functor of (\ref{eqn.ncdig}).

We now define the right vertical functor.  One can check explicitly that $C^{(2)}$ is one-periodic and that there is thus an induced equivalence
$$
{\sf proj }C^{(2)} \rightarrow {\sf proj }(\bigoplus_{j \geq 0}\operatorname{Hom}_{\sf H}(\tau^{j}\mathcal{L},\mathcal{L}))
$$
by Lemma \ref{lemma.veronese}.  We define the right vertical to be this equivalence.

The fact that (\ref{eqn.ncdig}) commutes is now a straightforward computation which we leave to the reader.

\subsection{Commutative conics from the noncommutative perspective} \label{subsec.conic}
Now suppose $\operatorname{char }k \neq 2$.  Let $a,b \in k$ be nonzero such that $(a,b)$ is the non-split quaternion algebra with $i^2=a$ and $j^2=b$.  Let $C(a,b)$ denote its associated conic in $\mathbb{P}^{2}_{k}$ defined as the zero set of $aX^2+bY^2-Z^2$.

In light of Theorem \ref{thm.newtsen}, the following result is implicit in \cite{kussintwo}.  For the readers convenience, we include a proof.
\begin{lemma} \label{lemma.prewitt}
The category ${\sf coh }C(a,b)$ is equivalent to ${\sf coh }\mathbb{P}^{nc}({{}_{(a,b)}}{(a,b)}_{k})$.
\end{lemma}

\begin{proof}
We begin by noting that since $C(a,b)$ has an indecomposable vector bundle, $\overline{\mathcal{L}}$, of rank two \cite[Theorem 2.2]{kussintwo}, Theorem \ref{thm.newtsen} implies that ${\sf coh }C(a,b)$ is $k$-equivalent to  ${\sf coh }\mathbb{P}^{nc}({}_{\operatorname{End}(\overline{\mathcal{L}})}\operatorname{End}(\overline{\mathcal{L}})_{k})$.  Now, the endomorphism rings of rank two indecomposable vector bundles over $C(a,b)$ are all isomorphic by \cite[Corollary 3.5]{kussintwo}.  Therefore, by \cite[Proposition 4.3]{kussintwo}, or by \cite[Lemma 3]{oldvdb} and the fact that $C(a,b)$ is the Brauer-Severi variety of $(a,b)$, we deduce $\operatorname{End}(\overline{\mathcal{L}}) \cong (a,b)$.  The result now follows from Lemma \ref{lemma.functorality} and Proposition \ref{prop.morita}.
\end{proof}

\begin{example} \label{example.quaternion}
Suppose $\operatorname{char }k \neq 2$, let $R := k$, and let $S := D$ be a non-split quaternion algebra over $k$.  Let $N := {}_{D}D_{k}$.  Then there exists a degree two extension $k'$ over $k$  such that $D_{k'}$ is split \cite[Proposition 1.2.3]{csa}.  It follows that if ${\sf Qcoh }\mathbb{P}^{1}_{k'}$ denotes the category of quasi-coherent sheaves over $\mathbb{P}^{1}_{k'}$, then
\begin{eqnarray*}
\mathbb{P}^{nc}(N)_{k'} & = & \mathbb{P}^{nc}(N_{k'}) \\
& \equiv & \mathbb{P}^{nc}(M_{2}(k')) \\
& \equiv & \mathbb{P}^{nc}(k' \oplus k') \\
& \equiv & {\sf Qcoh }\mathbb{P}^{1}_{k'}
\end{eqnarray*}
where the second equivalence is the equivalence from Lemma \ref{lemma.functorality} and Proposition \ref{prop.morita} induced by the isomorphism of bimodules
$$
{}_{D_{k'}}{D_{k'}}_{k'} \rightarrow {}_{M_{2}(k')}M_{2}(k')_{k'},
$$
the third equivalence is the equivalence from Proposition \ref{prop.morita} induced by the canonical equivalence
$$
G: M_{2}(k')-{\sf Mod} \rightarrow k'-{\sf Mod}
$$
sending an object $M$ to $e_{11}M$ and sending a morphism $f:M \rightarrow N$ to its restriction $e_{11}M \rightarrow e_{11}N$, and the last equivalence is from Lemma \ref{lemma.veronese}.
\end{example}

\section{Isomorphism invariants and noncommutative Witt's theorem} \label{section.equivs}
In \cite[Section 5]{nyman}, isomorphism invariants of noncommutative curves of genus zero whose underlying bimodule is a $(2,2)$-bimodule over a pair of isomorphic commutative fields are found.  In this section, using Theorem \ref{thm.newtsen}, we obtain a similar result for arbitrary noncommutative curves of genus zero (Theorem \ref{thm.equivs}), generalizing \cite[Proposition 5.1.4]{kussin}.  We will then use this result to obtain isomorphism invariants of noncommutative conics (those noncommutative curves of genus zero whose underlying bimodule is a $(1,4)$-bimodule).  This provides a generalization of Witt's theorem, as we shall see in Corollary \ref{cor.wittcor}.

\begin{theorem} \label{thm.equivs}
For $i=1,2$, let $D_{i}$ and $E_{i}$ be division rings finite dimensional over $k$, let $M$ be a $D_{1}-D_{2}$-bimodule of left right dimension $(2,2)$ or $(1,4)$ and let $N$ be an $E_{1}-E_{2}$-bimodule of left right dimension $(2,2)$ or $(1,4)$.  There is an equivalence
$$
{\sf H}_{1} = \mathbb{P}^{nc}(M) \longrightarrow \mathbb{P}^{nc}(N) = {\sf H}_{2}
$$
if and only if either
\begin{enumerate}
\item{} there exist isomorphisms $\phi_{i}:D_{i} \rightarrow E_{i}$ and $\psi:M \rightarrow N$ such that
$$
\psi(a \cdot m \cdot b) = \phi_{1}(a) \cdot \psi(m) \cdot \phi_{2}(b),
$$
or

\item{} there exist isomorphisms $\phi_{1}:D_{1} \rightarrow E_{2}$, $\phi_{2}:D_{2} \rightarrow E_{1}$ and $\psi:M \rightarrow N^{*}$ such that
$$
\psi(a \cdot m \cdot b) = \phi_{1}(a) \cdot \psi(m) \cdot \phi_{2}(b).
$$
\end{enumerate}
\end{theorem}

\begin{proof}
First, suppose there is an equivalence $F:\mathbb{P}^{nc}(M) \longrightarrow \mathbb{P}^{nc}(N)$, where $M=\operatorname{Hom}_{\sf H_{1}}(\mathcal{L},\overline{\mathcal{L}})$ and $N=\operatorname{Hom}_{\sf H_{2}}(\mathcal{L}',\overline{\mathcal{L}'})$.  Then $F$ induces an isomorphism of $k$-modules
\begin{equation} \label{eqn.finduced}
\operatorname{Hom}_{{\sf H}_{1}}(\mathcal{L},\overline{\mathcal{L}}) \rightarrow \operatorname{Hom}_{{\sf H}_{2}}(F(\mathcal{L}),F(\overline{\mathcal{L}})).
\end{equation}
We break the proof of the forward direction into two cases:
\newline
\newline
{\it Case 1}: $M$ is a $(1,4)$-bimodule.  By comparing ranks and using the fact that there are exactly two $\tau$-orbits of indecomposable bundles in ${\sf H}_{2}$ \cite[Section 1.1]{kussin}, we know there exist $i, j \in \mathbb{Z}$ such that
\begin{equation} \label{eqn.one}
F(\mathcal{L}) \overset{\cong}{\rightarrow} \tau^{i}\mathcal{L}'
\end{equation}
and
\begin{equation} \label{eqn.two}
F(\overline{\mathcal{L}}) \overset{\cong}{\rightarrow} \tau^{j}\overline{\mathcal{L}'}.
\end{equation}
Furthermore, since the map (\ref{eqn.finduced}) is a $k$-module isomorphism, it follows from \cite[Corollary 3.6]{tsen} that $i=j$.  Therefore $F$ together with the isomorphisms (\ref{eqn.one}) and (\ref{eqn.two}) induces $k$-algebra isomorphisms
$$
\phi_{1}: \operatorname{End}_{{\sf H}_{1}}(\overline{\mathcal{L}})
\rightarrow \operatorname{End}_{{\sf H}_{2}}(\tau^{i}\overline{\mathcal{L}'}) \overset{\tau^{-i}}{\rightarrow} \operatorname{End}_{{\sf H}_{2}}(\overline{\mathcal{L}'})
$$
and
$$
\phi_{2}: \operatorname{End}_{{\sf H}_{1}}(\mathcal{L}) \rightarrow \operatorname{End}_{{\sf H}_{2}}(\tau^{i}\mathcal{L}') \overset{\tau^{-i}}{\rightarrow} \operatorname{End}_{{\sf H}_{2}}(\mathcal{L}'),
$$
whose last maps send $f$ to $\eta (\tau^{-i}f )\eta^{-1}$, where $\eta \in \operatorname{Hom}_{{\sf H}_{2}}(\tau^{-i}\tau^{i}\overline{\mathcal{L}'},\overline{\mathcal{L}'})$ is induced by the adjointness of the pair $(\tau^{-1}, \tau)$.

Finally, we define $\psi$ as the composition
\begin{eqnarray*}
\operatorname{Hom}_{{\sf H}_{1}}(\mathcal{L},\overline{\mathcal{L}}) & \overset{F}{\rightarrow} & \operatorname{Hom}_{{\sf H}_{2}}(F(\mathcal{L}),F(\overline{\mathcal{L}})) \\
& \overset{\cong}{\rightarrow} & \operatorname{Hom}_{{\sf H}_{2}}(\tau^{i}\mathcal{L}', \tau^{i}\overline{\mathcal{L}'}) \\ & \overset{\tau^{-i}}{\rightarrow} & \operatorname{Hom}_{{\sf H}_{2}}(\mathcal{L}', \overline{\mathcal{L}'})
\end{eqnarray*}
whose second arrow is induced by (\ref{eqn.one}) and (\ref{eqn.two}).  It is straightforward to check that $\psi$, $\phi_{1}$ and $\phi_{2}$ satisfy the conclusion of the result.

The second possibility in this case is impossible since $N^{*}$ is a $(4,1)$-bimodule.
\newline
\newline
{\it Case 2:}  In the case that $M$ is a $(2,2)$-bimodule, the proof is similar, with one exception.  The module $F(\mathcal{L})$ can be any indecomposable bundle, since all are line bundles.  If $F(\mathcal{L}) \cong \tau^{i}\mathcal{L}'$, the proof proceeds as above.  Therefore, assume that there is an isomorphism
\begin{equation} \label{eqn.newone}
F(\mathcal{L}) \rightarrow \tau^{i}\overline{\mathcal{L}'}
\end{equation}
for some $i \in \mathbb{Z}$.  Then, as in the first case, there must be an isomorphism
\begin{equation} \label{eqn.newtwo}
F(\overline{\mathcal{L}}) \rightarrow \tau^{i-1}\mathcal{L}'.
\end{equation}
As in the first case, we define $\phi_{1}:\operatorname{End}_{{\sf H}_{1}}(\overline{\mathcal{L}}) \rightarrow \operatorname{End}_{{\sf H}_{2}}(\mathcal{L}')$ and $\phi_{2}:\operatorname{End}_{{\sf H}_{1}}({\mathcal{L}}) \rightarrow \operatorname{End}_{{\sf H}_{2}}(\overline{\mathcal{L}'})$, and we define $\psi:M \rightarrow N^{*}$ as the composition
\begin{eqnarray*}
\operatorname{Hom}_{{\sf H}_{1}}(\mathcal{L}, \overline{\mathcal{L}}) & \overset{F}{\rightarrow} & \operatorname{Hom}_{{\sf H}_{2}}(F(\mathcal{L}), F(\overline{\mathcal{L}})) \\
& \overset{\cong}{\rightarrow}  & \operatorname{Hom}_{{\sf H}_{2}}(\tau^{i}\overline{\mathcal{L}'}, \tau^{i-1}\mathcal{L}') \\
& \overset{(\tau^{i-1})^{-1}}{\rightarrow} & \operatorname{Hom}_{{\sf H}_{2}}(\tau \overline{\mathcal{L}'}, \mathcal{L}') \\
& \cong & N^{*}
\end{eqnarray*}
whose second arrow is induced by (\ref{eqn.newone}) and (\ref{eqn.newtwo}), and whose last isomorphism is the isomorphism of $\operatorname{End}_{{\sf H}_{2}}(\mathcal{L}')-\operatorname{End}_{{\sf H}_{2}}(\overline{\mathcal{L}'})$-bimodules from \cite[Proposition 3.4]{tsen}.  It is now routine to check that $\psi$, $\phi_{1}$ and $\phi_{2}$ satisfy the conclusion.

We now prove the converse.  First, suppose $M$ is a $(1,4)$-bimodule, so that there exist isomorphisms $\phi_{i}:D_{i} \rightarrow E_{i}$ and $\psi:M \rightarrow N$ such that $\psi(a \cdot m \cdot b) = \phi_{1}(a) \cdot \psi(m) \cdot \phi_{2}(b)$.  We let $\tilde{M}$ denote the $E_{1}-E_{2}$-bimodule with underlying set $M$ and left-right actions defined by
$$
c \cdot m \cdot d:= \phi_{1}^{-1}(c)m\phi_{2}^{-1}(d).
$$
By Proposition \ref{prop.morita}, we have an equivalence $\mathbb{P}^{nc}(M) \rightarrow \mathbb{P}^{nc}(\tilde{M})$.  Since $\psi$ induces an isomorphism $\tilde{M} \rightarrow N$ of $E_{1}-E_{2}$-bimodules, it follows from Lemma \ref{lemma.functorality} that $\psi$ induces a $k$-isomorphism of noncommutative symmetric algebras $\mathbb{S}^{nc}(\tilde{M}) \rightarrow \mathbb{S}^{nc}(N)$ and hence an equivalence $\mathbb{P}^{nc}(\tilde{M}) \rightarrow \mathbb{P}^{nc}(N)$.  The result follows in the $(1,4)$-case.

The converse in case $M$ is a $(2,2)$-bimodule is proven as above if we are in situation (1) in the statement of the theorem.  Otherwise, there exist isomorphisms $\phi_{1}:D_{1} \rightarrow E_{2}$, $\phi_{2}:D_{2} \rightarrow E_{1}$ and $\psi:M \rightarrow N^{*}$ such that $\psi(a \cdot m \cdot b) = \phi_{1}(a) \cdot \psi(m) \cdot \phi_{2}(b)$.  In this case, as above, there is an equivalence $\mathbb{P}^{nc}(M) \rightarrow \mathbb{P}^{nc}(N^{*})$.  On the other hand, by Corollary \ref{cor.mstar}, there is an equivalence $\mathbb{P}^{nc}(N^{*}) \rightarrow \mathbb{P}^{nc}(N)$, and the result follows.
\end{proof}

Specializing to the case of $(1,4)$-bimodules, we have the following:

\begin{cor} \label{cor.newwitt}
For $i=1,2$, let $D_{i}$ and $E_{i}$ be division rings finite dimensional over $k$, let $M$ be a $D_{1}-D_{2}$-bimodule of left right dimension $(1,4)$ and let $N$ be an $E_{1}-E_{2}$-bimodule of left right dimension $(1,4)$.  There is an equivalence
$$
\mathbb{P}^{nc}(M) \rightarrow \mathbb{P}^{nc}(N)
$$
if and only if there exist isomorphisms $\phi_{i}:D_{i} \rightarrow E_{i}$ and $\psi:M \rightarrow N$ such that
$$
\psi(a \cdot m \cdot b) = \phi_{1}(a) \cdot \psi(m) \cdot \phi_{2}(b).
$$
\end{cor}

We end this paper by describing how to use Corollary \ref{cor.newwitt} to recover Witt's theorem.  Retain the hypotheses and notation from Section \ref{subsec.conic}.  Let $D_{i}:=(a_{i},b_{i})$ for $i=1,2$ be non-split quaternion algebras.

\begin{cor} \label{cor.wittcor}
With the notation above $D_{1}$ is isomorphic to $D_{2}$ if and only if $C(a_{1},b_{1})$ is isomorphic to $C(a_{2},b_{2})$.
\end{cor}

\begin{proof}
First, suppose there is a $k$-algebra isomorphism $\phi_{1}:D_{1} \rightarrow D_{2}$. We define $\phi_{2}:k \rightarrow k$ to be the identity map and we let $\psi:D_{1} \rightarrow D_{2}$ equal $\phi_{1}$.  It follows that $\psi$ is compatible with $\phi_{1}$ and $\phi_{2}$ so that Corollary \ref{cor.newwitt} implies that there is an equivalence
$$
\mathbb{P}^{nc}({}_{D_{1}}{D_{1}}_{k}) \rightarrow \mathbb{P}^{nc}({}_{D_{2}}{D_{2}}_{k}).
$$
Therefore, by Lemma \ref{lemma.prewitt}, ${\sf coh }C(a_{1},b_{1})$ is equivalent to ${\sf coh }C(a_{2},b_{2})$ so that $C(a_{1},b_{1})$ is $k$-isomorphic to $C(a_{2},b_{2})$ by Rosenberg's reconstruction theorem \cite{reconstruction}.

Conversely, if $C(a_{1},b_{1})$ is $k$-isomorphic to $C(a_{2},b_{2})$, then by Lemma \ref{lemma.prewitt} again, there is an equivalence
$$
\mathbb{P}^{nc}({}_{D_{1}}{D_{1}}_{k}) \rightarrow \mathbb{P}^{nc}({}_{D_{2}}{D_{2}}_{k}),
$$
so that the result follows from Corollary \ref{cor.newwitt}.
\end{proof}


\begin{thebibliography}{11}



\bibitem{az} M. Artin and J. J. Zhang, Noncommutative projective schemes, {\it Adv. Math.} {\bf 109} (1994), no. 2, 228-287.

\bibitem{conics} I. Biswas and D.S. Nagaraj, Vector bundles over a nondegenerate conic, {\it J. Aust. Math. Soc.} {\bf 86} (2009), 145-154.


\bibitem{newchan} D. Chan and A. Nyman, Species and noncommutative $\mathbb{P}^1$'s over non-algebraic bimodules, {\it in progress}.

\bibitem{cohn} P. M. Cohn, Some remarks on the invariant basis property, {\it Topology} {\bf 5} (1966), 215-228.

\bibitem{csa} P. Gille and T. Szamuely, {\it Central Simple Algebras and Galois Cohomology} Cambridge Studies in Advanced Mathematics, vol. 101, Cambridge University Press, Cambridge, 2006.


\bibitem{kussintwo} D. Kussin, Factorial algebras, quaternions and preprojective algebras, {\it Algebras and modules, II (Geiranger, 1996),} 393-402, CMS Conf. Proc., 24, {\it Amer. Math. Soc., Providence, RI,} 1998.

\bibitem{kussin} D. Kussin, Noncommutative curves of genus zero: related to finite dimensional algebras, {\it Mem. Amer. Math. Soc.} {\bf 942} (2009), x+128pp.


\bibitem{lenzing} H. Lenzing and I. Reiten, Hereditary Noetherian categories of positive Euler characteristic, {\it Math. Z.} {\bf 254} (2006), 133-171.


\bibitem{decomp} I. Mori, On the Classification of Decomposable Quantum Ruled Surfaces, {\it Ring Theory 2007},
World Sci. Publ. (2009), 126-140.


\bibitem{tsen} A. Nyman, Noncommutative Tsen's theorem in dimension one, {\it J. Algebra} {\bf 434} (2015), 90-114.

\bibitem{duality} A. Nyman, Serre duality for non-commutative $\mathbb{P}^{1}$-bundles, {\it Trans. Amer. Math. Soc.} {\bf 357} (2005), 1349-1416.

\bibitem{nyman} A. Nyman, The geometry of arithmetic noncommutative projective lines, {\it J. Algebra} {\bf 414} (2014), 190-240.

\bibitem{reconstruction} A. L. Rosenberg, Reconstruction of Schemes, MPI Preprints Series, 1996
(108).


\bibitem{stafford} J.T. Stafford and M. Van den Bergh, Noncommutative curves and noncommutative surfaces, {\it Bull. Amer. Math. Soc.} (N.S.) {\bf 38} (2001), no. 2, 171-216.

\bibitem{oldvdb} M. Van den Bergh, The Brauer-Severi scheme of the trace ring of generic matrices, {\it Perspectives in ring theory} (Antwerp, 1987), 333-338, NATO Adv. Sci. Inst. Ser. C Math. Phys. Sci., 233, {\it Kluwer Acad. Publ., Dordrecht,} 1988.

\bibitem{vandenbergh}
M. Van den Bergh, Non-commutative $\mathbb{P}^{1}$-bundles over commutative schemes, {\it Trans. Amer. Math. Soc.} {\bf 364} (2012), 6279-6313.

\bibitem{quadrics}
M. Van den Bergh, Noncommutative quadrics, {\it Int. Math. Res. Not.} (2011), no. 17, 3983-4026.


\end{thebibliography}
\end{document}